\documentclass[a4paper, reqno]{amsart}

\usepackage[utf8]{inputenc}
\usepackage{eucal, stmaryrd}
\usepackage{mathtools}
\usepackage{graphicx}
\usepackage[colorlinks]{hyperref}
\usepackage{hyphenat}
\usepackage[table]{xcolor}

\newtheorem{lemma}{Lemma}[section]
\newtheorem{theorem}[lemma]{Theorem}

\newcommand{\Q}{\mathbb{Q}}
\newcommand{\R}{\mathbb{R}}

\newcommand{\Ical}{\mathcal{I}}

\newcommand{\Kcal}{\mathcal{K}}

\newcommand{\Tcal}{\mathcal{T}}

\newcommand{\sub}[2]{\mathrm{Sub}(#1, #2)}

\newcommand{\stb}{\Ical}
\newcommand{\clq}{\Kcal}

\newcommand{\cop}{\mathrm{COP}}           
\newcommand{\kpcone}{\mathcal{C}}         

\newcommand{\symg}{\mathfrak{S}}    
\newcommand{\tp}{\mathsf{T}}        
\newcommand{\csym}{C_{\rm sym}}     
\newcommand{\msym}{M_{\rm sym}}     
\newcommand{\one}{\mathbf{1}}       
\newcommand{\xip}{\xi}              

\newcommand{\flatten}[1]{\llbracket #1\rrbracket}
\newcommand{\invflatten}[2]{\llbracket #1\rrbracket_{#2}^{-1}}

\newenvironment{optprob}
{
  \arraycolsep=0pt
  \begin{array}{r@{\ }l@{\quad}l}
}%
{
  \end{array}
}

\newlength\claimlen

{
  \end{minipage}%
  \end{equation}%
  \ignorespacesafterend%
}

\newcommand{\defi}[1]{\textit{#1}}

\title{On the convergence of the $k$-point bound for topological packing graphs}

\author{Bram Bekker}
\address{A.J.F. Bekker, Delft Institute of Applied Mathematics,
  Delft University of Technology, Mekelweg~4, 2628~CD Delft, The
  Netherlands.}
\email{B.Bekker@tudelft.nl}

\author{Fernando Mário de Oliveira Filho}
\address{F.M. de Oliveira Filho, Delft Institute of Applied Mathematics,
  Delft University of Technology, Mekelweg~4, 2628~CD Delft, The
  Netherlands.}
\email{F.M.deOliveiraFilho@tudelft.nl}

\thanks{The first author is supported by the grant OCENW.KLEIN.024 of the Dutch
  Research Council (NWO)}

\subjclass[2010]{52C17, 90C22}

\date{June 5, 2023}

\begin{document}

\begin{abstract}
  We show that the $k$-point bound of de Laat, Machado, Oliveira, and
  Vallentin~\cite{LaatMOV2022}, a hierarchy of upper bounds for the independence
  number of a topological packing graph derived from the Lasserre hierarchy,
  converges to the independence number.
\end{abstract}

\maketitle
\markboth{A.J.F. Bekker and F.M. de Oliveira Filho}{On the convergence of the
  $k$-point bound for topological packing graphs}


\section{Introduction}%
\label{sec:intro}

Gvozdenović, Laurent, and Vallentin~\cite{GvozdenovicLV2009} introduced a
hierarchy of upper bounds for the independence number of a finite graph based on
a restriction of the Lasserre hierarchy.  Their hierarchy is weaker than the
Lasserre hierarchy but also easier to compute, and moreover it converges to the
independence number.

De Laat, Machado, Oliveira, and Vallentin~\cite{LaatMOV2022} later extended a
version of this hierarchy to topological packing graphs, calling it the
\defi{$k$-point bound}.  A graph is a \defi{topological graph} if its vertex set
is a Hausdorff space; we say that a topological graph is \defi{compact} if its
vertex set is compact.  A topological graph is a \defi{packing graph} if any
finite clique is contained in an open clique.

De Laat and Vallentin~\cite{LaatV2015} extended the Lasserre hierarchy to
topological packing graphs and showed convergence.  In this note, the same is
shown for the $k$-point bound (see Theorem~\ref{thm:k-point-conv} below) by
comparing it to a copositive hierarchy due to Kuryatnikova and
Vera~\cite{Kuryatnikova2019}.


\section{Preliminaries}

\subsection{Functional analysis}

Let~$X$ be a topological space.  We denote by~$C(X)$ the space of real-valued
continuous functions on~$X$ and by~$M(X)$ the space of signed Radon measures
on~$X$ of bounded total variation.  There is a duality between~$C(X)$ and~$M(X)$
which we always denote by
\[
  \langle f, \mu\rangle = \int_X f(x)\, d\mu(x)
\]
for~$f \in C(X)$ and~$\mu \in M(X)$.  Duals of cones and adjoints of operators
under this duality are denoted by an asterisk.

We denote by~$C(X)_{\geq 0}$ the cone of nonnegative functions in~$C(X)$.  Its
dual is the cone of nonnegative Radon measures, denoted by~$M(X)_{\geq 0}$.

By~$\csym(X)$ we denote the space of symmetric kernels~$F\colon X^2 \to \R$.
Similarly, we denote by~$\msym(X)$ the space of symmetric signed Radon measures
on~$X^2$, that is, the set of~$\mu \in M(X^2)$ such that~$\mu(U) = \mu(U^\tp)$
for~$U \subseteq X$, where~$U^\tp = \{\, (x, y) : (y, x) \in U\,\}$.

\subsection{Spaces of subsets}

Let~$V$ be a set.  For an integer~$k \geq 0$, denote by~$\sub{V}{k}$ the
collection of all subsets of~$V$ with cardinality at most~$k$.  For~$k \geq 1$
and~$v = (v_1, \ldots, v_k) \in V^k$, denote by~$\flatten{v}$ the
set~$\{v_1, \ldots, v_k\}$.  For every~$k \geq 1$ we have that~$\flatten{\cdot}$
maps~$V^k$ to~$\sub{V}{k}$.  Note that~$k$ is superfluous in the definition
of~$\flatten{\cdot}$, but not in the definition of the preimage.  Hence, given a
set~$S \subseteq \sub{V}{k}$, we write
\[
  \invflatten{S}{k} = \{\, v \in V^k : \flatten{v} \in S\,\}.
\]

If~$V$ is a topological space, then we can introduce
on~$\sub{V}{k} \setminus \{\emptyset\}$ the quotient topology
of~$\flatten{\cdot}$ by declaring a set~$S$ open if~$\invflatten{S}{k}$ is open
in~$V^k$.  We define a topology on~$\sub{V}{k}$ by taking the disjoint union
with~$\{\emptyset\}$.  For background on the topology on the space of subsets,
see Handel~\cite{Handel2000}.

Let~$G = (V, E)$ be a topological graph and~$k \geq 0$ be an integer.
By~$\stb_k$ we denote the collection of all independent sets of~$G$ of
cardinality at most~$k$, which as a subset of~$\sub{V}{k}$ can be equipped with
the relative topology.  We denote by~$\stb_{=k}$ the collection of all
independent sets of cardinality~$k$.  We have the following
lemma~\cite[Lemma~2]{LaatV2015}.

\begin{lemma}%
  \label{lem:ind-topo}
  If~$G = (V, E)$ is a compact topological packing graph, then~$\stb_{=r}$ is
  both closed and open in~$\stb_k$ for all~$r \leq k$.  If~$Y$ is a topological
  space, then~$f\colon \stb_k \to Y$ is continuous if and only if the
  restriction of~$f$ to~$\stb_{=r}$ is continuous for all~$r \leq k$.
\end{lemma}

The edge set of a topological packing graph is open.  We will prove a stronger
result, not found in the literature, that includes the converse of this
assertion.

Handel~\cite[Proposition~2.11]{Handel2000} showed that if~$U_1$, \dots,~$U_r$
are disjoint open sets in~$V$ and if~$k \geq r$, then the set
\[
  (U_1, \ldots, U_r)_k = \{\, A \in \sub{V}{k} : \text{$A \cap U_i \neq
    \emptyset$ for all~$i$ and $A \subseteq U_1 \cup \cdots \cup U_r$}\,\}
\]
is open.  The collection of all such sets for~$r \leq k$ forms a basis for the
topology on~$\sub{V}{k}$.

Let~$\clq_k$ denote the set of cliques of~$G$ of cardinality at most~$k$, and
likewise let~$\clq_{=k}$ denote the set of cliques of cardinality~$k$, both
equipped with the relative topology induced by~$\sub{V}{k}$.  We can
identify~$E$ with~$\clq_{=2}$.  We will commonly switch between considering~$E$
as a subset of~$\sub{V}{2}$ or as a symmetric subset of~$V^2$, since it does not
make a difference for the topology; it will be clear from context in which
situation we are.

\begin{theorem}
  For a topological graph~$G = (V, E)$, the following are equivalent:
  \begin{enumerate}
  \item[(i)] $G$ is a packing graph;
    
  \item[(ii)] $\clq_k$ is open in~$\sub{V}{k}$ for every~$k \geq 0$;
    
  \item[(iii)] $\clq_2$ is open in~$\sub{V}{2}$.
  \end{enumerate}
\end{theorem}

\begin{proof}
To see that~(i) implies~(ii), let~$C = \{x_1, \ldots, x_k\}$ be a clique of
cardinality~$k$.  It is contained in an open clique~$K$.  Since~$V$ is a
Hausdorff space, there are disjoint open sets~$U_1$, \dots,~$U_k$ such
that~$x_i \in U_i$ for all~$i$.  By taking the intersection of the~$U_i$
with~$K$, we can assume that the union of the~$U_i$ is a clique.  The
set~$(U_1, \ldots, U_k)_k$ is an open set in~$\sub{V}{k}$ containing~$C$ and
consisting only of cliques.  This proves that~$\clq_k$ is open in~$\sub{V}{k}$.

That~(ii) implies~(iii) is immediate.  To see that~(iii) implies~(i), we first
prove that every vertex lies in an open clique.

Assume~$\clq_2$ is open in~$\sub{V}{2}$.  Since the sets of the form~$(U_1,
U_2)_2$ with~$U_1$, $U_2$ open and disjoint form a basis of the topology
of~$\sub{V}{2}$, we can assume that for~$x \in V$ we have an open
neighborhood~$(U)_2$ of~$\{x\}$ for some~$U$ open such that~$(U)_2 \subseteq
\clq_2$.  So~$U$ is an open clique containing~$x$.

Let~$C \subseteq V$ be a finite clique.  We claim that any
pair~$\{x, y\} \subseteq C$ is contained in an open clique.  Indeed, there are
disjoint open sets~$U_x \ni x$ and~$U_y \ni y$ such
that~$(U_x, U_y)_2 \subseteq \clq_2$.  In particular, any choice of a
pair~$(x', y') \in U_x \times U_y$ is an edge.  We know that both~$x$ and~$y$
are contained in open cliques~$C_x$ and~$C_y$, respectively, and
so~$U_x \cap C_x$ and~$U_y \cap C_y$ are disjoint open cliques.  We see
that~$(U_x \cap C_x) \cup (U_y \cap C_y)$ is an open clique
containing~$\{x, y\}$, as we wanted.  Write~$U_{x,y} = U_x \cap C_x$
and~$U_{y, x} = U_y \cap C_y$ and let~$U_{x,x}$ be any open clique
containing~$x$.

For every~$x \in C$, the set~$\bigcap_{y \in C} U_{x,y}$ is an open clique
containing~$x$.  Moreover, for all~$z \in C$ every pair of distinct vertices in
$U_{z, x} \times \bigcap_{y \in C} U_{x,y}$ is an edge.  This shows that
\[
  \bigcup_{x \in C} \bigcap_{y \in C} U_{x, y}
\]
is an open clique containing~$C$.
\end{proof}


\section{The \texorpdfstring{$k$}{k}-point bound}

Let~$G = (V, E)$ be a topological packing graph.  For~$k \geq 2$,
let~$\csym(\stb_1^2 \times \stb_{k-2})$ be the space of continuous
functions~$F\colon \stb_1^2 \times \stb_{k-2} \to \R$ such
that~$F(S, T, Q) = F(T, S, Q)$ for all~$(S, T, Q)$.  Likewise,
let~$\msym(\stb_1^2 \times \stb_{k-2})$ be the space of signed Radon
measures~$\nu$ such that~$\nu(X) = \nu(X^\tp)$ for
all~$X \subseteq \stb_1^2 \times \stb_{k-2}$, where
$X^\tp = \{\, (S, T, Q) : (T, S, Q) \in X\,\}$.

Let~$B_k\colon \csym(\stb_1^2 \times \stb_{k-2}) \to C(\stb_k)$ be such that for
every function~$F$ and every~$I \in \stb_k$ we have
\begin{equation}%
  \label{eq:bk-def}
  (B_k F)(I) = \sum_{Q \in \sub{V}{k-2}} \sum_{\substack{S, T \in \sub{I}{1}\\Q
      \cup S \cup T = I}} F(S, T, Q).
\end{equation}
We prove in~\S\ref{sec:bk-cont} that~$B_k F$ is indeed continuous.  The
map~$B_k$ is linear and bounded, and hence continuous.  So we can take its
adjoint, which is the
operator~$B_k^*\colon M(\stb_k) \to \msym(\stb_1^2 \times \stb_{k-2})$.

Let~$C(\stb_1^2 \times \stb_{k-2})_{\succeq 0}$ be the cone of continuous
functions~$F \in \csym(\stb_1^2 \times \stb_{k-2})$ such that the kernel
$(S, T) \mapsto F(S, T, Q)$ is positive semidefinite for every~$Q$.  We denote
the dual of this cone by~$M(\stb_1^2 \times \stb_{k-2})_{\succeq 0}$.

For an integer~$k \geq 2$, the \defi{$k$-point bound} for~$G$ is
\[
  \begin{optprob}
    \Delta_k(G) = \sup&\nu(\stb_{=1})\\
                      &\nu(\{\emptyset\})=1,\\
                      &B_k^{*}\nu\in M(\stb_1^2\times \stb_{k-2})_{\succeq0},\\
                      &\nu\in M(\stb_{k})_{\geq0}.
  \end{optprob}
\]
We denote both this problem and its optimal value by~$\Delta_k(G)$.  There are
two ways in which this definition deviates from the one by de Laat, Machado,
Oliveira, and Vallentin~\cite{LaatMOV2022}.  First, the original formulation
excludes the empty set from~$\stb_k$.  Second, a different normalization is
used.  Including the empty set is necessary in the proof of convergence given
in~\S\ref{sec:convergence}; it seems to give stronger, nonequivalent problems.
Changing the normalization can be done without affecting convergence, and the
proof presented in~\S\ref{sec:convergence} relies on this fact.

We get a hierarchy of bounds, namely
\[
  \Delta_2(G) \geq \Delta_3(G) \geq \cdots \geq \alpha(G).
\]
Indeed, we can easily restrict a solution of~$\Delta_{k+1}(G)$ to a solution
of~$\Delta_k(G)$ with the same objective value.  Moreover, if~$I$ is any
independent set of~$G$, then
\[
  \nu = \sum_{R \in \stb_k,\ R \subseteq I} \delta_R,
\]
where~$\delta_R$ is the Dirac measure at~$R$, is a feasible solution
of~$\Delta_k(G)$ with objective value~$|I|$, and
so~$\Delta_k(G) \geq \alpha(G)$.


\subsection{Continuity of~\texorpdfstring{$B_k$}{Bk}}%
\label{sec:bk-cont}

It is not true in general that~$B_k F$ is continuous, but this is the case
when~$G$ is a \textsl{packing} graph.  This fact is used in the literature, but
no proof is to be found, so here is one.

\begin{theorem}%
  \label{thm:bk-cont}
  If~$G = (V, E)$ is a topological packing graph, if~$k \geq 2$ is an integer,
  and if~$F \in \csym(\stb_1^2 \times \stb_{k-2})$, then~$B_k F$ is continuous.
\end{theorem}

We need the following simple lemma.

\begin{lemma}%
  \label{lem:set-net}
  If~$V$ is a topological Hausdorff space, if~$(S_\alpha)$ is a net
  in~$\sub{V}{k}$ that converges to~$S$, and if~$U$ is an open set such
  that~$S \cap U \neq \emptyset$, then there is~$\alpha_0$ such that~$S_\alpha
  \cap U \neq \emptyset$ for all~$\alpha \geq \alpha_0$.
\end{lemma}

\begin{proof}
If the statement is not true, then there is an open set~$U$ with~$S \cap U \neq
\emptyset$ such that for all~$\alpha_0$ there is~$\alpha \geq \alpha_0$
with~$S_\alpha \cap U = \emptyset$.

Since~$\sub{V}{k}$ is closed, we know that~$|S| \leq k$.
Set~$W = \flatten{U \times V^{k-1}}$, which~\cite[Lemma~2.6]{Handel2000} is an
open set in~$\sub{V}{k}$.  It contains all subsets of~$V$ of cardinality at
most~$k$ that contain at least one element of~$U$, and so~$S \in W$.
If~$S_\alpha$ is such that~$S_\alpha \cap U = \emptyset$,
then~$S_\alpha \notin W$.  But then for all~$\alpha_0$ there
is~$\alpha \geq \alpha_0$ with~$S_\alpha \notin W$, a contradiction.
\end{proof}

\begin{proof}[Proof of Theorem~\ref{thm:bk-cont}]
By Lemma~\ref{lem:ind-topo} it suffices to show that~$B_k F$ is continuous
on~$\stb_{=r}$ for every~$0 \leq r \leq k$.  To do so, fix~$r$ and let~$(I_\alpha)$
be a net in~$\stb_{=r}$ that converges to~$I \in \stb_{=r}$; we show
that~$(B_k F)(I_\alpha)$ converges to~$(B_k F)(I)$.

To see this, say~$I = \{x_1, \ldots,x_r\}$ and take disjoint open sets~$U_1$,
\dots,~$U_r$ such that~$x_i \in U_i$.  Applying Lemma~\ref{lem:set-net} to
each~$U_i$ and taking an upper bound we see that there is~$\alpha_0$ such that
for all~$\alpha \geq \alpha_0$ we have~$I_\alpha \cap U_i \neq \emptyset$ for
all~$i$.  Hence, since the~$U_i$ are disjoint, we can write
$I_\alpha = \{x_{\alpha, 1}, \ldots, x_{\alpha, r}\}$
with~$x_{\alpha,i} \in U_i$ for all~$\alpha \geq \alpha_0$.

It follows that, for~$\alpha \geq \alpha_0$, the double sum in~\eqref{eq:bk-def}
for~$I = I_\alpha$ can be written in terms of the indices~$1$, \dots,~$r$
instead of the elements of~$I_\alpha$ themselves.  More precisely, the number of
terms in the double sum is always the same, and there is a natural bijection
between the summands of~$(B_k F)(I_\alpha)$ and of~$(B_k F)(I_\beta)$ such that
between two corresponding summands the points~$x_{\alpha, i}$ and~$x_{\beta, i}$
appear in the same places.  Since~$(x_{\alpha, i})$ converges to~$x_i$ for
all~$i$, and since~$F$ is continuous, the theorem follows.
\end{proof}


\section{The copositive hierarchy}

Let~$V$ be a finite set.  A symmetric matrix~$A \in \R^{V \times V}$ is
\defi{copositive} if~$x^\tp A x \geq 0$ for all~$x \in \R^V$ with~$x \geq 0$.
The set of all copositive matrices on~$V$, denoted by $\cop(V)$, is a closed
convex cone.

Let~$V$ be a topological space. A continuous kernel~$F \in \csym(V)$ is
\defi{copositive} if~$\bigl(F(x, y)\bigr)_{x,y \in U}$ is a copositive matrix
for every finite set~$U \subseteq V$.  The set of all copositive kernels on~$V$,
denoted by~$\cop(V)$, is a convex cone as well.

For a finite graph~$G$, if the cone of positive-semidefinite matrices is
replaced by the cone of copositive matrices in the dual (minimization) version
of the Lovász theta number semidefinite program, then the optimal value of the
resulting problem is exactly~$\alpha(G)$.  This result goes back to Motzkin and
Straus~\cite{MotzkinS1965}, as observed by de Klerk and
Pasechnik~\cite{KlerkP2002a}, and was extended to topological packing graphs by
Dobre, Dür, Frerick, and Vallentin~\cite{DobreDFV2016}.

Since solving copositive programs is computationally intractable, hierarchies of
approximations of the copositive cone have been proposed.  One such hierarchy
was introduced by de Klerk and Pasechnik~\cite{KlerkP2002a}, and later extended
by Kuryatnikova and Vera~\cite{Kuryatnikova2019} to topological packing graphs.

Let~$\symg_n$ be the symmetric group on~$n$ elements and let~$\Tcal_r\colon
C(V^2) \to C(V^{r+2})$ be the operator
\[
  (\Tcal_r F)(x_1, \ldots, x_{r+2}) = \frac{1}{(r+2)!} \sum_{\pi\in\symg_{r+2}}
  (F \otimes \one^{\otimes r})(x_{\pi(1)}, \ldots, x_{\pi(r+2)}),
\]
where~$\one$ is the constant one function.  Note
that~$(F \otimes \one^{\otimes r})(x_1, \ldots, x_{r+2}) = F(x_1, x_2)$.
Let~$r \geq 0$ be an integer and define
\[
  \kpcone_r(V) = \{\, F \in \csym(V) : \Tcal_r F \geq 0\,\}.
\]
This is a closed convex cone.  Kuryatnikova and
Vera~\cite[Theorem~2.9]{Kuryatnikova2019} show that
\[
  \kpcone_1(V) \subseteq \kpcone_2(V) \subseteq \cdots \subseteq \cop(V)
\]
and that any copositive kernel in the algebraic interior of~$\cop(V)$ belongs
to~$\kpcone_r(V)$ for some~$r$.

This gives a converging hierarchy of upper bounds for the independence number.
Namely, let~$G = (V, E)$ be a compact topological packing graph and for
integer~$r \geq 0$ consider the optimization problem
\begin{equation}%
  \label{opt:xi-p-dual}
  \begin{optprob}
    \xip_r(G)^* = \inf&\lambda\\
    &Z(x, x) \leq \lambda - 1&\text{for all~$x \in V$,}\\
    &Z(x, y) \leq -1&\text{for all distinct~$x$, $y \in V$ with~$(x, y) \notin
      E$,}\\
    &Z \in \kpcone_r(V).
  \end{optprob}
\end{equation}
Here, the asterisk is used to emphasize that this is the dual of the problem in
which we are interested.

We have~$\xip_r(G)^* \geq \alpha(G)$ for all~$r \geq 0$, and so we get a
hierarchy of upper bounds for the independence number, that is,
\[
  \xip_1(G)^* \geq \xip_2(G)^* \geq \cdots \geq \alpha(G).
\]
Kuryatnikova and Vera~\cite[Theorem~2.17]{Kuryatnikova2019} showed that this
hierarchy converges\footnote{The theorem presented in Kuryatnikova's
  thesis~\cite{Kuryatnikova2019} relies on an extra assumption which is however
  unnecessary; see Appendix~\ref{sec:cop-conv-general}.}.

\begin{theorem}%
  \label{thm:kpcone-conv}
  If~$G$ is a compact topological packing graph, then~$\xip_r(G)^* \to
  \alpha(G)$ as~$r \to \infty$.
\end{theorem}

The dual of~\eqref{opt:xi-p-dual}, in the sense of
Barvinok~\cite[Chapter~IV]{Barvinok2002}, is
\begin{equation}%
  \label{opt:xi-p-primal}
  \begin{optprob}
    \xip_r(G) = \sup&\alpha(\Delta)\\
    &\alpha(\{\emptyset\}) = 1,\\
    &\alpha_E = 0,\\
    &\alpha = \Tcal_r^* \beta,\\
    &\alpha \in \msym(V)_{\geq 0},\ \beta \in M(V^{r+2})_{\geq 0}.
  \end{optprob}
\end{equation}
Here,~$\Delta = \{\, (x, x) : x \in V\,\}$ is the diagonal, which is closed and
therefore measurable.  By~$\alpha_E$ we denote the restriction of~$\alpha$
to~$E$; since~$E$ is open, this restriction is again a signed Radon measure.
Note also that $\kpcone_r(V)^* = \Tcal_r^* M(V^{r+2})_{\geq 0}$,
so~$\alpha \in \kpcone_r(V)^*$.

\begin{theorem}%
  \label{thm:k-point-conv}
  If~$G$ is a compact topological packing graph, then for every~$r \geq 1$ we
  have~$\Delta_{r+2}(G) \leq \xip_r(G)$ and~$\Delta_k(G) = \alpha(G)$ for
  all~$k \geq \alpha(G) + 2$.
\end{theorem}


\section{Convergence of the \texorpdfstring{$k$}{k}-point bound}%
\label{sec:convergence}

Throughout this section,~$G = (V, E)$ is a compact topological packing graph.
Our goal is to prove Theorem~\ref{thm:k-point-conv}, and we do so by taking a
feasible solution of~$\Delta_{r+2}(G)$ and making a feasible solution
of~\eqref{opt:xi-p-primal} with the same objective value.

Let~$r \geq 0$ be an integer and let~$S \subseteq V$.  Denote by~$N_r(S)$ the
number of tuples~$v \in V^r$ such that~$\flatten{v} = S$.  We use the
convention~$V^0 = \{\emptyset\}$ and~$\flatten{\emptyset} = \emptyset$, so the
definition includes the case~$r = 0$.  We do not specify a domain for~$N_r$; we
will use it as a function on~$\stb_k$ for any~$k$ we need.  Note that~$N_r(S)$
depends only on the cardinality of~$S$, and so by Lemma~\ref{lem:ind-topo} the
function~$N_r\colon \stb_k \to \R$ is continuous.

Fix integers~$k \geq 0$, $s \geq 1$, and~$0 \leq t \leq s$.  Let~$Q_{s,t}\colon
C(V^t) \to C(\stb_k)$ be the map such that
\[
  (Q_{s,t}F)(I) = \sum_{\substack{v \in V^s\\\flatten{v} = I}} F(v_1, \ldots,
  v_t).
\]
A proof similar to that of Theorem~\ref{thm:bk-cont} shows that~$Q_{s,t}F$ is
indeed continuous for every continuous~$F$.  Of course, these maps also depend
on~$k$, but since~$k$ does not appear in the expression above, we omit it.

\begin{lemma}%
  \label{lem:final}
  Let~$k$, $s$, and~$t$ be integers such that~$k \geq 0$, $s \geq 1$,
  and~$0 \leq t \leq s$ and let~$G = (V, E)$ be a compact topological packing
  graph.  We have:
  \begin{enumerate}
  \item[(i)] The map~$Q_{s,t}$ is continuous and~$(Q_{s,t}^*\nu)(V^t) = \langle
    N_s, \nu\rangle$ for all~$\nu \in M(\stb_k)$. 
    
  \item[(ii)] If~$t + 2 \leq s$, then~$Q_{s,t+2} \Tcal_t = Q_{s,2}$ and
    hence~$\Tcal_t^* Q_{s,t+2}^* = Q_{s,2}^*$.
  \end{enumerate}
\end{lemma}
  
\begin{proof}
The map~$Q_{s,t}$ is linear and bounded and hence continuous, so its
adjoint is well defined.  Moreover, if~$\nu \in M(\stb_k)$ then
\[
  \begin{split}
    (Q_{s,t}^*\nu)(V^t)
    &= \langle Q_{s,t}\chi_{V^t}, \nu\rangle\\
    &=\int_{\stb_k} \sum_{\substack{v \in V^s\\\flatten{v} = I}} 1\, d\nu(I)\\
    &=\langle N_s, \nu\rangle,
  \end{split}
\]
proving~(i).

To see~(ii), let~$F \in C(V^2)$.  For every~$I \in \stb_k$ we have
\[
  \begin{split}
    (Q_{s, t+2} \Tcal_t F)(I)
    &=\sum_{\substack{v \in V^s\\\flatten{v} = I}} \frac{1}{(t+2)!} \sum_{\pi \in
    \symg_{t+2}} F(v_{\pi(1)}, v_{\pi(2)})\\ 
    &=\frac{1}{(t+2)!} \sum_{\pi \in \symg_{t+2}}
      \sum_{\substack{v \in V^s\\\flatten{v} = I}} F(v_{\pi(1)}, v_{\pi(2)})\\ 
    &=\sum_{\substack{v \in V^s\\\flatten{v} = I}} F(v_1, v_2)\\
    &=(Q_{s,2} F)(I),
  \end{split}
\]
and~(ii) follows.
\end{proof}

\begin{proof}[Proof of Theorem~\ref{thm:k-point-conv}]
Given a feasible solution of~$\Delta_{r+2}(G)$ we will construct a feasible
solution of~$\xip_r(G)$ with the same objective value,
whence~$\Delta_{r+2}(G) \leq \xip_r(G)$.  Since we know by weak duality
that~$\xip_r(G) \leq \xip_r(G)^*$, together with Theorem~\ref{thm:kpcone-conv}
we see that~$\Delta_k(G) \to \alpha(G)$ as~$k \to \infty$.
Since~$\stb_k = \stb_{\alpha(G)}$ for all~$k \geq \alpha(G)$, we
have~$B_k = B_{\alpha(G) + 2}$ for all~$k \geq \alpha(G) + 2$, whence
$\Delta_k(G) = \Delta_{\alpha(G) + 2}(G)$ for all~$k \geq \alpha(G) + 2$, and
the theorem follows.

Fix~$r$ and let~$\nu$ be a feasible solution of~$\Delta_{r+2}(G)$ with positive
objective value.  We will write~$Q_t = Q_{r+2, t}$ for short, where in the
definition of~$Q_t$ we take~$k = r+2$.

For~$t \geq 0$ write~$\Phi_t = \langle N_t, \nu\rangle$.  We will see later
that~$\Phi_t > 0$ for all~$t$; for now, let us assume this to be true.  Write
\[
  \beta = \Phi_{r+1}^{-1} Q_{r+2}^* \nu \in M(V^{r+2})\qquad\text{and}\qquad
  \alpha = \Tcal_r^* \beta \in \msym(V).
\]
We claim that~$(\alpha, \beta)$ is feasible for~$\xip_r(G)$.

To begin, note that if~$F \in C(V^{r+2})$ is nonnegative, then so
is~$Q_{r+2} F$, hence
$\langle F, Q_{r+2}^* \nu\rangle = \langle Q_{r+2} F, \nu\rangle \geq 0$
and~$\beta$ is nonnegative.  If~$F \in \csym(V)$ is nonnegative, then so
is~$\Tcal_r F$, whence
$\langle F, \alpha\rangle = \langle F, \Tcal_r^*\beta\rangle = \langle \Tcal_r
F, \beta\rangle \geq 0$, and~$\alpha$ is nonnegative.

Next, if~$F\in \csym(V)$ is a function with support contained in~$E$, then using
Lemma~\ref{lem:final} we get
\[
  \begin{split}
    \Phi_{r+1} \langle F, \alpha\rangle
    &= \langle F, \Tcal_r^* Q_{r+2}^* \nu\rangle\\
    &= \langle F, Q_2^* \nu\rangle\\
    &= \int_{\stb_{r+2}} \sum_{\substack{v \in V^{r+2}\\\flatten{v} = I}} F(v_1,
    v_2)\, d\nu(I)\\
    &= 0.
  \end{split}
\]
Since~$E$ is open, it is itself a locally compact Hausdorff space,
and~$\alpha_E$ is a signed Radon measure on~$E$.  It then follows from the Riesz
representation theorem that~$\alpha_E = 0$.

We now want to calculate~$\alpha(\Delta)$.  Every vertex is contained in an open
clique, so by compactness there are open cliques~$C_1$, \dots,~$C_n$ whose union
is~$V$.  Now the set~$U = C_1^2 \cup \cdots \cup C_n^2$ is an open set in~$V^2$
whose union contains~$\Delta$.  Moreover, if~$(x, y) \in U$ and~$x \neq y$,
then~$(x, y) \in E$, that is,~$U \setminus \Delta \subseteq E$.

Urysohn's lemma gives us a continuous function~$F\colon V^2 \to [0, 1]$ that
is~$1$ on~$\Delta$ and~$0$ outside of~$U$.  Since~$\alpha_E = 0$, using
Lemma~\ref{lem:final} we get
\[
  \begin{split}
    \Phi_{r+1} \alpha(\Delta)
    &= \Phi_{r+1} \langle F, \alpha\rangle\\
    &= \langle F, \Tcal_r^* Q_{r+2}^* \nu\rangle\\
    &= \langle F, Q_2^*\nu\rangle\\
    &=\int_{\stb_{r+2}} \sum_{\substack{v \in V^{r+2}\\\flatten{v} = I}} F(v_1,
    v_2)\, d\nu(I)\\
    &=\int_{\stb_{r+2}} \sum_{\substack{v \in V^{r+1}\\\flatten{v} = I}} 1\,
    d\nu(I)\\
    &=\Phi_{r+1},
  \end{split}
\]
whence~$\alpha(\Delta) = 1$. 

To finish the proof we only need to show that~$\Phi_t > 0$ and
that~$\alpha(V^2) \geq \nu(\stb_{=1})$, and this follows from the claim:
if~$t \leq r + 1$, then the matrix
\begin{equation}%
  \label{eq:2x2-mat}
  \begin{pmatrix}
    \Phi_t&\Phi_{t+1}\\
    \Phi_{t+1}&\Phi_{t+2}
  \end{pmatrix}
\end{equation}
is positive semidefinite.

Indeed, assume the claim.  Since~$\Phi_0 = 1$ and~$\Phi_1 = \nu(\stb_{=1}) > 0$,
we immediately get~$\Phi_2 > 0$.  Repeating this argument we get~$\Phi_t > 0$
for all~$t$.

As for the objective value, use Lemma~\ref{lem:final} to get
\[
  \alpha(V^2) = \langle \chi_{V^2}, \Phi_{r+1}^{-1} \Tcal_r^* Q_{r+2}^*
  \nu\rangle
  = \langle Q_2 \chi_{V^2}, \nu\rangle \Phi_{r+1}^{-1}
  = \Phi_{r+2} \Phi_{r+1}^{-1}.
\]
Since~\eqref{eq:2x2-mat} is positive semidefinite, we
have~$\Phi_{t+2} \Phi_{t+1}^{-1} \geq \Phi_{t+1} \Phi_t^{-1}$.  Repeated
application of this inequality
yields~$\Phi_{r+2} \Phi_{r+1}^{-1} \geq \Phi_1 \Phi_0^{-1} = \nu(\stb_{=1})$,
and so~$\alpha(V^2) \geq \nu(\stb_{=1})$, as we wanted.

We prove that~\eqref{eq:2x2-mat} is positive semidefinite.  A
measure~$\mu \in \msym(\stb_1)$ is \defi{positive semidefinite}
if~$\langle F, \mu\rangle \geq 0$ for every positive-semidefinite
kernel~$F \in \csym(\stb_1)$.  We claim that there is such a positive
semidefinite measure~$\mu$ such that
$\mu(\stb_{=i} \times \stb_{=j}) = \Phi_{t+i+j}$ for~$i$, $j = 0$, $1$.

To see this, use the Riesz representation theorem to let~$\mu$ be the measure
such that $\langle F, \mu\rangle = \langle F \otimes N_t, B_{r+2}^* \nu\rangle$
for every~$F \in C(\stb_1^2)$, where we choose~$\stb_r$ as the domain
of~$N_t$.  

To see that~$\mu$ is positive semidefinite, let~$F \in \csym(\stb_1)$ be a
positive-semidefinite kernel.  Since~$N_t \geq 0$, the kernel
$(S, T) \mapsto F(S, T) N_t(Q)$ is positive semidefinite for every~$Q$.
So~$F \otimes N_t \in C(\stb_1^2 \times \stb_r)_{\succeq 0}$
and~$\langle F, \mu\rangle \geq 0$
since~$B_{r+2}^* \nu \in M(\stb_1^2 \times \stb_r)_{\succeq 0}$.

Let us now calculate~$\mu(\stb_{=i} \times \stb_{=j})$ for~$i$, $j = 0$, $1$.
For every~$I \in \stb_{r+2}$ and~$i = j = 0$ we have
\[
  \begin{split}
    (B_{r+2}(\chi_{\{\emptyset\}}^{\otimes 2} \otimes N_t))(I)
    &= \sum_{Q \in \sub{I}{r}} \sum_{\substack{S, T \in \sub{I}{1}\\Q \cup S
    \cup T = I}} \chi_{\{\emptyset\}}(S) \chi_{\{\emptyset\}}(T) N_t(Q)\\
    &= N_t(I).
  \end{split}
\]

For~$i = 0$ and~$j = 1$ we get
\[
  \begin{split}
    (B_{r+2}(\chi_{\{\emptyset\}} \otimes \chi_{\stb_{=1}} \otimes N_t))(I)
    &= \sum_{Q \in \sub{I}{r}} \sum_{\substack{x \in I\\Q \cup \{x\} = I}}
    N_t(Q)\\
    &= \sum_{Q \in \sub{I}{t}} \sum_{\substack{x \in I\\Q \cup \{x\} = I}}
    \sum_{\substack{v \in V^t\\\flatten{v} = Q}} 1.
  \end{split}
\]
There is a bijection between the set of
triples~$(Q, x, v) \in \sub{I}{t} \times I \times V^t$ such
that~$Q \cup \{x\} = I$ and~$\flatten{v} = Q$ and the set of
tuples~$v \in V^{t+1}$ such that~$\flatten{v} = I$,
namely~$(Q, x, v) \leftrightarrow (x, v)$.  Hence
\[
  (B_{r+2}(\chi_{\{\emptyset\}} \otimes \chi_{\stb_{=1}} \otimes N_t))(I) =
  \sum_{\substack{v \in V^{t+1}\\\flatten{v} = I}} 1 = N_{t+1}(I).
\]

Similarly, for~$i = j = 1$ we have
\[
  \begin{split}
    (B_{r+2}(\chi_{\stb_{=1}}^{\otimes 2} \otimes N_t))(I)
    &= \sum_{Q \in \sub{I}{r}} \sum_{\substack{x, y \in I\\Q \cup \{x, y\} =
    I}} N_t(Q)\\
    &= \sum_{Q \in \sub{I}{t}} \sum_{\substack{x, y \in I\\Q \cup \{x, y\} = I}}
    \sum_{\substack{v \in V^t\\\flatten{v} = Q}} 1\\
    &= \sum_{\substack{v \in V^{t+2}\\\flatten{v} = I}} 1\\
    &= N_{t+2}(I).
  \end{split}
\]

Putting it all together, we get that~\eqref{eq:2x2-mat} is equal to
\[
  A = \begin{pmatrix}
    \mu(\stb_{=0}^2)&\mu(\stb_{=0} \times \stb_{=1})\\
    \mu(\stb_{=0} \times \stb_{=1})&\mu(\stb_{=1}^2)
  \end{pmatrix}.
\]
This matrix is positive semidefinite, since for~$x_0$, $x_1 \in \R$ we have
\[
  (x_0, x_1)^\tp A (x_0, x_1) = \int_{\stb_1^2} (x_0 \chi_{\stb_{=0}} + x_1
  \chi_{\stb_{=1}})^{\otimes 2}(S, T)\, d\mu(S, T) \geq 0,
\]
and we are done.
\end{proof}



\appendix

\section{Convergence of the copositive hierarchy}%
\label{sec:cop-conv-general}

Kuryatnikova and Vera~\cite[Theorem~2.17]{Kuryatnikova2019} proved
Theorem~\ref{thm:kpcone-conv} under an extra assumption, namely the existence of
a kernel~$Z_0$ in the algebraic interior of~$\cop(V)$ such
that~$Z_0(x, y) \leq -1$ for all~$(x, y) \in \overline{E}$, where
\[
  \overline{E} = \{\, (x, y) \in V^2 : \text{$x \neq y$ and~$(x, y) \notin
    E$}\,\}.
\]
They also use a slightly different version of~\eqref{opt:xi-p-dual}, namely they
require~$Z(x, x) = \lambda - 1$ for all~$x \in V$.  We show next that the
hierarchy~\eqref{opt:xi-p-dual} converges by showing that a~$Z_0$ as above
always exists.

\begin{proof}[Proof of Theorem~\ref{thm:kpcone-conv}]
Let~$G = (V, E)$ be a compact topological packing graph.  Note first
that~$\xip_r(G)^* \geq \alpha(G)$ for all~$r \geq 1$.  Indeed, if the problem is
infeasible, then~$\xip_r(G)^* = \infty$ and we are done.  If~$(\lambda, Z)$ is a
feasible solution of~$\xip_r(G)^*$ and if~$I$ is a nonempty independent set,
then since~$Z$ is copositive we have
\[
  0 \leq \sum_{x, y \in I} Z(x, y) \leq |I|(\lambda - 1) - (|I|^2 - |I|),
\]
whence~$|I| \leq \lambda$.
So~$\lim_{r\to\infty} \xip_r(G)^* = L \geq \alpha(G)$.

We claim that there is~$Z_0$ in the algebraic interior of~$\cop(V)$ such
that~$Z_0(x, y) \leq -1$ for~$(x, y) \in \overline{E}$.  Indeed, de Laat and
Vallentin~\cite[Lemma~7]{LaatV2015} showed that there is a positive-semidefinite
kernel~$F$ satisfying~$F(x, y) \leq -1$ for all~$(x, y) \in \overline{E}$.
Then~$2F$ is positive semidefinite and~$2F(x, y) \leq -2$ for
all~$(x, y) \in \overline{E}$, and we can get our kernel~$Z_0$ by making a
convex combination between the constant one kernel~$J$, which is in the
algebraic interior of~$\cop(V)$, and~$2F$.

Let~$\xip_\infty(G)^*$ be the problem obtained from~\eqref{opt:xi-p-dual} by
replacing~$\kpcone_r(V)$ by the copositive cone~$\cop(V)$.  Dobre, Dür, Frerick,
and Vallentin~\cite{DobreDFV2016} showed that~$\xip_\infty(G)^* = \alpha(G)$.
Let~$(\lambda, Z)$ be any feasible solution of~$\xip_\infty(G)^*$.  For
every~$0 < \epsilon \leq 1$, the kernel~$\epsilon Z_0 + (1 - \epsilon) Z$ is in
the algebraic interior of~$\cop(V)$, and so~\cite[Theorem~2.9]{Kuryatnikova2019}
it belongs to~$\kpcone_r(V)$ for some~$r$.  Hence, by taking~$\epsilon \to 0$,
we see that~$L \leq \lambda$.  Since~$(\lambda, Z)$ is an arbitrary feasible
solution of~$\xip_\infty(G)^*$, we are done.
\end{proof}


\begin{thebibliography}{9}
\bibitem{Barvinok2002}
A.~Barvinok, {\it A Course in Convexity}, Graduate Studies in
Mathematics~54, American Mathematical Society, Providence, Rhode
Island, 2002.

\bibitem{DobreDFV2016}
C.~Dobre, M.E.~Dür, L.~Frerick, and F.~Vallentin, A copositive
formulation for the stability number of infinite graphs, {\it
Mathematical Programming, Series~A\/}~160 (2016) 65--83.

\bibitem{GvozdenovicLV2009}
N.~Gvozdenović, M.~Laurent, and F.~Vallentin, Block-diagonal
semidefinite programming hierarchies for 0/1 programming, {\it
Operations Research Letters\/}~37 (2009) 27--31.

\bibitem{Handel2000}
D.~Handel, Some homotopy properties of spaces of finite subsets of
topological spaces, {\it Houston Journal of Mathematics\/}~26 (2000)
747--764.

\bibitem{KlerkP2002a}
E.~de Klerk and D.V.~Pasechnik, Approximation of the stability number
of a graph via copositive programming, {\it SIAM Journal on
Optimization\/}~12 (2002) 875--892.

\bibitem{Kuryatnikova2019}
O.~Kuryatnikova, {\it The many faces of positivity to approximate
structured optimization problems}, Ph.D. Thesis, Tilburg University,
2019.

\bibitem{LaatMOV2022}
D.~de Laat, F.C.~Machado, F.M.~de Oliveira Filho, and F.~Vallentin,
$k$-point semidefinite programming bounds for equiangular lines, {\it
Mathematical Programming, Series~A\/}~194 (2022) 533--567.

\bibitem{LaatV2015}
D.~de Laat and F.~Vallentin, A semidefinite programming hierarchy for
packing problems in discrete geometry, {\it Mathematical Programming,
Series~B\/}~151 (2015) 529--553.

\bibitem{MotzkinS1965}
T.S.~Motzkin and E.G.~Straus, Maxima for graphs and a new proof of a
theorem of Turán, {\it Canadian Journal of Mathematics\/}~17 (1965)
533--540.

\end{thebibliography}
\end{document}